\documentclass[11pt,oneside,a4paper]{article}
\usepackage{graphicx}
\usepackage[intlimits]{amsmath}
\usepackage{amsthm}
\usepackage{amsmath}
\usepackage{amssymb}
\usepackage[T1]{fontenc}
\usepackage{color}

\newtheorem{thm}{Theorem}
\newtheorem{lem}{Lemma}
\newtheorem{rem}[lem]{Remark}
\newtheorem{defin}[lem]{Definition}
\newtheorem{cor}[lem]{Corollary}
\newtheorem{example}{Example}
\newtheorem{prop}[lem]{Proposition}

\def \pK{\mathcal{K}}
\def \pP{\mathcal{P}}
\def \pN{\mathcal{N}}
\def \RR{\mathbb{R}}
\def \NN{\mathbb{N}}
\def \Rd{\RR^d}

\def \tp{\tilde{p}}
\def \ap{p^a}

\author{
Tomasz Jakubowski\footnote{The research was partially supported by ANR-09-BLAN-0084-01 and grants MNiSW N N201 397137 and N N201 422539. Fellowship co-financed by European Union within European Social Fund.}, Karol Szczypkowski\\
\scriptsize Institute of Mathematics, Wroc\l{}aw University of Technology, Wyb. Wyspia\'nskiego 27,\\
\scriptsize 50-370 Wroc\l{}aw, Poland \\
\scriptsize tomasz.jakubowski@pwr.wroc.pl, \quad \scriptsize karol.szczypkowski@pwr.wroc.pl}
\title{Estimates of gradient perturbation series}
\date{}

\begin{document}
\maketitle

\begin{abstract}
We give upper and lower bounds of perturbation series for transition densities, corresponding to additive gradient perturbations  satisfying certain space-time integrability conditions.
\end{abstract}
\begin{quote}
 {\em keywords: }transition density, gradient perturbations.
\end{quote}

\section{Introduction}

Perturbation series is one of few explicit methods to construct new semigroups and it is widely used in many areas of mathematics and physics. In this paper we study the perturbation series in the context of  gradient perturbations of transition densities on $\RR^d$, $d \in \NN_+$. A function $p \colon \RR\times \RR^d\times\RR\times \RR^d \to [0,\infty)$ is called a transition density if for all $-\infty<s<t<\infty$ and $x,y \in \Rd$ it satisfies the Chapman-Kolmogorov equation,
\begin{align}\label{assume1}
\int_{\Rd} p(s,x,u,z)p(u,z,t,y)\,dz = p(s,x,t,y)\,,\quad \rm{s<u<t}.
\end{align}
The function $p$ may describe the distribution of particles evolving in space and time. We are interested in adding a drift to the picture. Let $b=(b_i)_{i=1}^{d}\colon\RR \times\Rd \to\Rd$ (the drift function). The perturbation series is
\begin{equation}\label{def:tp}
\tp(s,x,t,y)=\sum_{n=0}^{\infty}p_n(s,x,t,y)\,,
\end{equation}
where $p_{0}(s,x,t,y)=p(s,x,t,y)$ and for $n=1,2,...$
\begin{align}
p_n(s,x,t,y)&=\int_s^t \int_{\mathbb{R}^d} p_{n-1}(s,x,u,z)b(u,z)\cdot \nabla_{z}p(u,z,t,y)dzdu\,.
\label{def:p_n}
\end{align}

We will focus on the convergence and estimates of the perturbation series. 
The series (\ref{def:tp}) is motivated by the perturbation theory of semigroup operators. Namely, if we denote by $L$ the generator of the time-inhomogeneous semigroup with the integral density $p$, then, heuristically, $\tp$ stands for the density of the evolution generated by $L + b \cdot \nabla$. This observation and the series (\ref{def:tp}) were used in the study of gradient perturbations of elliptic operators (e.g. \cite{Zhang2}, \cite{MR1783642}) and the fractional Laplacian (\cite{BJ}, \cite{JS}). In such approach the convergence of (\ref{def:tp}) is delicate and therefore suitable conditions on $b$ need to be assumed. The general philosophy is to state the conditions in terms of the density $p$ in such a way that they fit the iteration procedure (\ref{def:p_n}). This lead to the relative Kato conditions for Schr\"odinger perturbations in \cite{BHJ} and \cite{J4}. We note that there exist probabilistic methods to study Schr\"odinger perturbations  based on multiplicative functionals and Khasminski's lemma (strengthened in \cite{BHJ}). However gradient perturbations are more delicate and probabilistic methods (e.g. Girsanov transform) are applicable in special situations.

In the present paper we will consider natural conditions (\ref{def:coN}) and (\ref{def:coP}) analogous to those used in  the case of Schr\"odinger perturbations in \cite{BHJ} and \cite{J4}. 

\begin{defin}\label{def:N}
Let $\eta \ge 0$ and $Q \colon \RR \times \RR \to [0,\infty)$ satisfy 
\begin{equation}\label{eq:Q}
Q(r,u) + Q(u,v) \le Q(r,v)\,, \qquad r<u<v\,.
\end{equation}
We say that $b \in \pN(\eta,Q,p)$ if for all $-\infty < s<t<\infty$ and $x,y \in \Rd$,
\begin{equation}\label{def:coN}
\int_s^t \int_{\Rd} p(s,x,u,z)|b(u,z)||\nabla_z p(u,z,t,y)|\,dz\,du \le \big[ \eta + Q(s,t)\big] p(s,x,t,y)\,.
\end{equation}
\end{defin}

\begin{defin}\label{def:P}
Let $\eta>0$. We will say that $b \in \pP (\eta,p)$ if there exists $h>0$ such that for all $t-s<h$ and $x,y\in \Rd$,
\begin{align}
\int_s^t \int_{\Rd} p(s,x,u,z)|b(u,z)||\nabla_z p(u,z,t,y)|\,dzdu \leq \eta p(s,x,t,y)\,. \label{def:coP}
\end{align}
\end{defin}

If $\eta$ or $Q$ are not specified, by writing $b \in \pN(\eta,Q,p)$ we mean that (\ref{def:coN}) is satisfied for some $\eta$ and $Q$.

As a part of Definition \ref{def:N} and \ref{def:P} we will always make the following assumption on the gradient of $p$: for all $x,y \in \Rd$ and $s<u<t$,
\begin{align}
& \nabla_x p(s,x,t,y) \text{\; exists}, \,\, \text{and} \nonumber\\
& \nabla_x p(s,x,t,y) = \int_{\Rd} \nabla_x p(s,x,u,z)p(u,z,t,y)\, dz \,,\label{assume2}
\end{align}
where the integral is absolutely convergent.
\begin{rem}
\begin{rm}\label{rem:coP}
If $b\in \pP(\eta, p)$, then $b\in \pN (\eta, \beta (t-s),p)$ with $\beta=\frac{\eta}{h}$, where $h$ is taken from Definition \ref{def:P} (see \cite{BHJ}, \cite{JS}). We will generally state our results for the larger class $\pN$, but occasionally more specific results will be given for $\pP$.
\end{rm}
\end{rem}

The main results of the paper are the following two theorems.
\begin{thm}\label{thm:2}
Let $p$ be a function satisfying (\ref{assume1}) and (\ref{assume2}) and $b\in \pN(\eta,Q,p)$ with $\eta<\frac{1}{2}$. Then the perturbations series (\ref{def:tp}) converges absolutely, there exists a constant $C>1$ such that for all $-\infty<s<t<\infty$ and $x,y \in \Rd$,
\begin{equation}\label{eq:thm1}
\frac{p(s,x,t,y)}{C^{1 + Q(s,t)}}  \leq \tp(s,x,t,y) \leq 
p(s,x,t,y) 
\begin{cases}
\left( \frac{1}{1-2\eta}\right)^{1+\frac{Q(s,t)}{\eta}},&  {\rm if} \quad 0<\eta<\frac{1}{2}\,,\\
&\\
e^{Q(s,t)},&  {\rm if}
 \quad \eta=0\,,
\end{cases}
\end{equation}
and the Chapman-Kolmogorov equation holds for $\tp$,
\begin{equation}\label{eq:Ch}
\tp(s,x,t,y)=\int_{\Rd} \tp(s,x,u,z)\tp(u,z,t,y)\,dz\,, \quad u\in (s,t)\,.
\end{equation}
\end{thm}

Similar results were first obtained in \cite{JS} for the density $p$ of the isotropic $\alpha$-stable process ($1<\alpha<2$). The authors considered drift functions $b$  satisfying the following condition (see also \cite{Zhang2}),
\begin{equation}\label{def:coK}
\int_s^t \int_{\Rd} \left(\frac{p(s,x,u,z)}{(u-s)^{1/\alpha}} + \frac{p(u,z,t,y)}{(t-u)^{1/\alpha}}\right)|b(u,z)|\,dz\,du \le \eta + Q(s,t)\,,
\end{equation}
where $\eta\geq 0$, and $Q$ satisfies (\ref{eq:Q}). 
We note that the condition (\ref{def:coN}) is more natural and general than (\ref{def:coK}), e.g. it allows the drift $|b(u,z)| = p(0,0,u,z)^{(\alpha-1)/d}$  (see \cite[Remark 7, Example 5]{JS}). 

In this paper we propose a new general method which may be applied to various functions $p$. As an example in section 3 we consider the density $p$ of the semigroup generated by $L = \Delta^{\alpha/2}+a^{\beta}\Delta^{\beta/2}$ ($1<\beta<\alpha<2$, $a>0$), and prove that the resulting $\widetilde{p}$ is the density of the semigroup corresponding to $L + b \cdot \nabla$. This result is stated in the following theorem (see section 3 for detailed definitions)

\begin{thm}\label{thm:4}
Let $1<\beta<\alpha<2$ and $p^a(s,x,t,y)$ be the density of the semigroup generated by $\Delta^{\alpha/2} + a^\beta \Delta^{\beta/2}$. If $b\in \pN (\eta,Q, \ap)$ with $\eta<1/2$, then there exists a transition density $\widetilde{\ap}$ satisfying (\ref{eq:thm1}) and such that
\begin{align*}
\int_s^{\infty}\int_{\Rd}\widetilde{\ap}(s,x,u,z) \Big[&\partial_u \phi(u,z) + \left( \Delta^{\alpha/2}_z + a^{\beta} \Delta^{\beta/2}_z \right) \phi(u,z) +\\
&b(u,z) \cdot \nabla_{z}\phi(u,z)\Big]\,dzdu\,= -\phi(s,x)\,,
\end{align*}
where $s\in \RR$, $x \in \Rd$ and $\phi \in C_{c}^{\infty}(\RR \times \Rd )$.
\end{thm}

\section{Proofs}
Throughout this section we assume that $\eta\geq 0$ and $Q$ is a function satisfying (\ref{eq:Q}). The following lemma is taken from \cite{JS} (see also \cite{BHJ}). For convenience of the reader we attach the proof.
\begin{lem}\label{lem:p_n_chap}
Let $b\in \pN (\eta,Q,p)$.
For all $s<u<t$, $x,y\in\Rd$ and $n=0,1,2,\ldots$, we have
\begin{equation}\label{eq:p_n_chap}
\sum_{m=0}^{n} \int_{\mathbb{R}^d}p_{m}(s,x,u,z)p_{n-m}(u,z,t,y)dz=p_{n}(s,x,t,y)\,.
\end{equation}
\end{lem}

\begin{proof}
By (\ref{assume1}), (\ref{eq:p_n_chap}) is true for $n=0$. Let $n\ge 1$ and assume that (\ref{eq:p_n_chap}) holds for $n-1$. By Fubini's theorem, (\ref{assume1}) and (\ref{assume2}), the last term in the above sum equals
\begin{align*}
&\int_{\mathbb{R}^d} p_{n}(s,x,u,z)p_{0}(u,z,t,y)\,dz \\
= &\int_{\mathbb{R}^d} \int_s^u\int_{\mathbb{R}^d}p_{n-1}(s,x,r,w)b(r,w)\cdot \nabla_{w}p(r,w,u,z)dwdr\ p(u,z,t,y)dz \\
= &\int_s^u\int_{\mathbb{R}^d}p_{n-1}(s,x,r,w)b(r,w)\cdot \nabla_{w}p(r,w,t,y)dwdr\,.
\end{align*}
\noindent
Furthermore,
\begin{align*}
&\sum_{m=0}^{n-1}\int_{\mathbb{R}^d}p_{m}(s,x,u,z)p_{n-m}(u,z,t,y)dz \\
= &\sum_{m=0}^{n-1}\int_{\mathbb{R}^d}p_{m}(s,x,u,z)\int_u^t\int_{\mathbb{R}^d}
p_{n-1-m}(u,z,r,w)b(r,w)\cdot \nabla_{w}p(r,w,t,y)dwdr\,dz \\
= &\int_u^t\int_{\mathbb{R}^d}\Bigg(\sum_{m=0}^{n-1}\int_{\mathbb{R}^d}p_{m}(s,x,u,z)
p_{n-1-m}(u,z,r,w)dz\Bigg)b(r,w)\cdot \nabla_{w}p(r,w,t,y)dwdr \\
= &\int_u^t\int_{\mathbb{R}^d}p_{n-1}(s,x,r,w)b(r,w)\cdot \nabla_{w}p(r,w,t,y)dwdr\,.
\end{align*}
The two observations yield (\ref{eq:p_n_chap}).
\end{proof}

The following lemma is crucial in our consideration. It will allow us to sum the series (\ref{def:tp}) regardless of the smallness of $Q(s,t)$.
\begin{lem}\label{lem:comb}
Let $s<t$, $k\in \mathbb{N}_+$, $\theta >0$, and $s=t_0<t_1<\ldots <t_k=t$, be such that for all $i= 0,1,\ldots,k-1 $ and $x,y\in \Rd$
\begin{equation}\label{eq:comb}
\int_{t_i}^{t_{i+1}} \int_{\Rd} p(t_i,x,u,z)|b(u,z)||\nabla_z p(u,z,t_{i+1},y)|\,dzdu \leq \theta\, p(t_i,x,t_{i+1},y)\,.
\end{equation}
Then for all $x,y\in\Rd$,
\begin{align*}
|p_n(s,x,t,y)|\leq \binom{n+k-1}{k-1} \theta^n p(s,x,t,y)\,.
\end{align*} 
\end{lem}

\begin{proof}
For $k=1$ the inequality is true for every $n$ by the definition of $p_n$ and induction in $n$. If $k>1$ and the statement is true for $k-1$ and all $n$, then for every $n\in \mathbb{N}$, by Lemma \ref{lem:p_n_chap} we obtain
$$
|p_n(s,x,t,y)|\leq \sum_{m=0}^n \int_{\Rd} |p_m(s,x,t_1,z)| |p_{n-m}(t_1,z,t,y)|\,dz
$$
$$
\leq \sum_{m=0}^n \int_{\Rd} \theta^m p(s,x,t_1,z) \binom{n-m+k-2}{k-2}\theta^{n-m}p(t_1,z,t,y)\,dz
$$
$$
\leq \sum_{m=0}^{n} \binom{n-m+k-2}{k-2} \theta^n p(s,x,t,y)
$$
$$
= \binom{n+k-1}{k-1} \theta^n p(s,x,t,y)\,.
$$
\end{proof}

We note that the function $Q$ may be discontinuous. If $Q$ has a jump bigger than $\theta$ then (\ref{eq:comb}) does not hold for any partition of the interval $(s,t)$. We will overcome this problem by replacing $Q(s,t)$ in (\ref{def:coN}) by $\lim_{h \to 0} Q(s_0,t-h)- Q(s_0,s+h)$ for some $s_0 \le s$. 
We will write, as usual,
\begin{align*}
F(t^-)=\lim_{u \uparrow t} F(u)\,, \qquad F(s^+)=\lim_{u \downarrow s} F(u)\,.
\end{align*}

\begin{lem} \label{lem:F}
Let $b\in \pN (\eta,Q,p)$ and $s_0\in\RR$.  Define $F(u)=Q(s_0,u)$ if $u>s_0$ and $F(u)=0$ if $u\leq s_0$. Then for all $s_0\leq s<t<\infty$ and $x,y\in\Rd$,
$$
\int_s^t \int_{\Rd} p(s,x,u,z)|b(u,z)||\nabla_z p(u,z,t,y)|\, dzdu \leq \left( \eta + F(t^-)-F(s^+) \right) p(s,x,t,y)\,.
$$
\end{lem}

\begin{proof}
For all $s_0\leq s<t<\infty$ and $x,y\in \Rd$ we have
\begin{align*}
\int_s^t \int_{\Rd} p(s,x,&u,z)|b(u,z)||\nabla_z p(u,z,t,y)|\, dzdu \\
= &\lim_{h\to 0^+}\int_{s+h}^{t-h} \int_{\Rd} p(s,x,u,z)|b(u,z)||\nabla_z p(u,z,t,y)|\, dzdu \\
\leq &\limsup_{h\to 0^+} \int_{\Rd} \int_{\Rd} p(s,x,s+h,w_2) \int_{s+h}^{t-h} \int_{\Rd} p(s+h,w_2,u,z) |b(u,z)| \\
& \qquad \qquad |\nabla_z p(u,z,t-h,w_2)|\,dzdu\, p(t-h,w_1,t,y)\, dw_1 dw_2\\
\leq &\lim_{h \to 0^+} \big[ \eta +Q(s+h,t-h) \big] p(s,x,t,y)\\
\leq &\lim_{h\to 0^+}\big[ \eta + F(t-h)-F(s+h) \big] \, p(s,x,t,y)\\
= & \big[ \eta + F(t^-)-F(s^+)\big] \, p(s,x,t,y)\,,
\end{align*}
because $Q(s+h,t-h)\leq Q(s_0,t-h)-Q(s_0,s+h)$.
\end{proof}

\begin{lem}\label{lem:funckja}
Let $F\colon \RR \to [0,\infty)$ be non-decreasing. Let $\theta>0$, $s<t$ and $k\in \mathbb{N}_+$ be such that $F(t^-)-F(s^+)\leq k\theta$. Then there are $m \in \mathbb{N}_+$, $m\leq k$, $s=t_0< t_1 <\ldots < t_m =t$ such that $F(t_{i+1}^-)-F(t_i^+) \leq \theta$ for $i= 0,\ldots ,m-1$. 
\end{lem}

\begin{proof}
Let $l\in\mathbb{N}_+$ be the smallest number such that $F(t^-)-F(s^+)\leq l\theta$. If $l=1$ we take $t_0 =s$, $t_1 =t$. Otherwise, we define $r_0=s$, $r_{l} =t$ and $r_i = \sup \{ u\in (s,t) \colon F(u)-F(s^+)\leq i \theta \}$ for $i = 1,\ldots,l-1$. We note that $F(r_{i+1}^-)-F(s^+)\leq (i+1)\theta$ and $F(r_i^+)-F(s^+)\geq i\theta$ for $i= 0,\ldots ,l-1$, hence 
\begin{align*}
F(r_{i+1}^-)-F(r_i^+) &= \left(F(r_{i+1}^-)-F(s^+) \right) - \left( F(r_i^+)-F(s^+)\right)\\
&\leq (i+1)\theta - i\theta = \theta\,.
\end{align*}
Now let $m+1$ be the number of the elements of the set $R = \{r_0,\ldots,r_l\}$. We put $t_0 = r_0$ and $t_k = \min\{r_i \in R\colon r_i > t_{k-1}\}$, for $t =1, \ldots,m$.
\end{proof}

Now we are ready to prove Theorem \ref{thm:2}.
\begin{proof}[Proof of Theorem \ref{thm:2}]
Let $s<t$ and $s_0 =s$. Let $F(u)=Q(s_0,u)$ if $u>s_0$, and $F(u)=0$ if $u\leq s_0$.\\
We will prove the upper bound of (\ref{eq:thm1}) first.\\
Let $\varepsilon>0$ and $k\in \mathbb{N}_+$ be such that $(k-1)\varepsilon \leq F(t^-)-F(s^+)\leq k\varepsilon$. By Lemma \ref{lem:funckja} there are $m\in \mathbb{N}_+$, $m\leq k$ and $s=t_0<t_1<\ldots <t_m =t$ such that $F(t_{i+1}^-)-F(t_i^+)\leq \varepsilon$ for $i=0,\ldots ,m-1$. By Lemma \ref{lem:F} and Lemma \ref{lem:comb}  with $\theta = \eta+\varepsilon$, for all $x,y \in \Rd$ we obtain
\begin{align*}
\tp(s,x,t,y)&\leq \sum_{n=0}^{\infty} |p_n(s,x,t,y)|\leq \sum_{n=0}^{\infty} \binom{n+m-1}{m-1}\left(\eta+\varepsilon\right)^n p(s,x,t,y)\\
& = \left( \frac{1}{1-(\eta+\varepsilon)} \right)^m p(s,x,t,y) \leq \left( \frac{1}{1-(\eta+\varepsilon)} \right)^k p(s,x,t,y)\\
&\leq \left( \frac{1}{1-(\eta+\varepsilon)} \right)^{1+ \frac{F(t^-)-F(s^+)}{\varepsilon}} p(s,x,t,y)\,.
\end{align*}
We put $\varepsilon=\eta$ when $\eta>0$ and we let $\varepsilon\to 0$ when $\eta=0$. 
The above calculation justifies the last inequality in the statement in the theorem, as well as the change of the order of the integration and the use of Cauchy product in what follows. By Lemma \ref{lem:p_n_chap},
\begin{align*}
 \int_{\Rd}\tp(s,x,u,z)\tp(u,z,t,y)\,dz &=\int_{\Rd}\sum_{n=0}^{\infty} \sum_{m=0}^{n} p_m(s,x,u,z)p_{n-m}(u,z,t,y)\,dz \\
&=\sum_{n=0}^{\infty} \sum_{m=0}^{n} \int_{\Rd} p_m(s,x,u,z)p_{n-m}(u,z,t,y)\,dz \\
&=\sum_{n=0}^{\infty}p_n(s,x,t,y)=\tp(s,x,t,y)\,. 
\end{align*}
We now prove the lower bound. Let $\delta \in (0, \frac{1}{2}-\eta)$.\\
If $F(t^-)-F(s^+)\leq \frac{1}{2}-\eta-\delta$, then by Lemma \ref{lem:F} and induction 
$$|p_n(s,x,t,y)| \leq \left( \eta + F(t^-)-F(s^+) \right)^n p(s,x,t,y)\leq \left( \frac{1}{2}-\delta \right)^n p(s,x,t,y)\,,$$
and we get 
\begin{align}
\tp(s,x,t,y) &\geq p(s,x,t,y)-\sum_{n=1}^{\infty}|p_n(s,x,t,y)| \nonumber\\
&\geq \left( 1-\sum_{n=1}^{\infty}\left( \frac{1}{2}-\delta \right)^n \right)p(s,x,t,y) \nonumber\\
&=\left( 1- \frac{\frac{1}{2}-\delta}{1-\frac{1}{2}+\delta}\right) p(s,x,t,y) = \left( \frac{4\delta}{1+2\delta} \label{eq:low} \right)p(s,x,t,y)\,. 
\end{align}
Now in general, we set $k\in \mathbb{N}_+$ such that $(k-1) \left(  \frac{1}{2}-\eta -\delta \right) \leq F(t^-)-F(s^+)\leq k \left( \frac{1}{2}-\eta -\delta \right) $. By Lemma \ref{lem:funckja} there are $m\in \mathbb{N}_+$, $m\leq k$ and $s=t_0<t_1<\ldots <t_m =t$ such that $F(t_{i+1}^-)-F(t_i^+)\leq \frac{1}{2}-\eta-\delta$.  By (\ref{eq:low}) and (\ref{eq:Ch}) for all $x,y \in \Rd$,
\begin{align*}
\tp(s,x,t,y) &\geq \left( \frac{4\delta}{1+2\delta} \right)^m p(s,x,t,y) \geq \left( \frac{4\delta}{1+2\delta} \right)^k p(s,x,t,y)\\
&\geq \left( \frac{4\delta}{1+2\delta} \right)^{1 + \frac{F(t^-)-F(s^+)}{\frac{1}{2}-\eta-\delta}}p(s,x,t,y)\,.
\end{align*}
The assertion is true with $C=\left( 4\delta/(1+2\delta) \right)^{-\frac{1}{\frac{1}{2}-\eta-\delta}}$\,.
\end{proof}

If $b \in  \pP(\eta,p)$, then the proof is simpler and the estimates are better.

\begin{thm}\label{thm:3}
Let $p(s,x,t,y)$ be a function satisfying (\ref{assume1}) and (\ref{assume2}) and $b\in \pP(\eta,p)$ with $\eta<\frac{1}{2}$. Then for all $-\infty<s<t<\infty$, $x,y\in \Rd$,
\begin{equation}\label{ineq:thm2}
\left(\frac{1-\eta}{1-2\eta}\right)^{-1 - \frac{\beta(t-s)}{\eta}} p(s,x,t,y) \leq \tp(s,x,t,y)\leq \left( \frac{1}{1-\eta} \right)^{1+ \frac{\beta(t-s)}{\eta}} p(s,x,t,y)
\end{equation}\\
and the Chapman-Kolmogorov equation (\ref{eq:Ch}) holds.
\end{thm}
\begin{proof}
We have (\ref{eq:Ch}) by Theorem \ref{thm:2} and Remark \ref{rem:coP}. For the proof of (\ref{ineq:thm2}) we let $s<t$ and $k\in\mathbb{N}_+$ be such that $(k-1)h\leq t-s<kh$. By Lemma \ref{lem:comb} with $t_i=s+\frac{t-s}{k}i$ and $\theta=\eta$,
\begin{align*}
\tp(s,x,t,y)&\leq \sum_{n=0}^{\infty} |p_n(s,x,t,y)|\leq \sum_{n=0}^{\infty} \binom{n+k-1}{k-1}\eta^n p(s,x,t,y)\\
&= \left( \frac{1}{1-\eta} \right)^k p(s,x,t,y)\leq \left(\frac{1}{1-\eta} \right)^{1+\frac{\beta(t-s)}{\eta}} p(s,x,t,y)\,,
\end{align*}
where $\beta=\frac{\eta}{h}$.\\
Also, if $t-s\leq h$, then 
\begin{align}
\tp(s,x,t,y) &\geq p(s,x,t,y)-\sum_{n=1}^{\infty}|p_n(s,x,t,y)| \notag\\
&\geq \left( 1-\sum_{n=1}^{\infty}\eta^n \right)p(s,x,t,y) =\left(\frac{1-2\eta}{1-\eta}\right) p(s,x,t,y)\,. \label{eq1:thm3}
\end{align}
Now, for any $s<t$ and $k\in \mathbb{N}_+$ such that $(k-1)h\leq t-s<kh$, by (\ref{eq1:thm3}) and (\ref{eq:Ch}) we obtain
\begin{align*}
\tp(s,x,t,y) &\geq \left(\frac{1-2\eta}{1-\eta} \right)^{k}p(s,x,t,y)\\
&\geq \left(\frac{1-2\eta}{1-\eta} \right)^{1+\frac{\beta(t-s)}{\eta}}p(s,x,t,y)\,.
\end{align*}
\end{proof}

\section*{The case of the mixed fractional Laplacian}
As an example we consider the transition density of the Brownian motion subordinated by the sum of two independent stable subordinators. Such processes were recently studied in \cite{CKS}.
Let $a\geq 0$, $0<\beta<\alpha<2$. Denote $\ap(s,x,t,y)=\ap(t-s,y-x)$, where $-\infty < s<t<\infty$, $x,y \in\Rd$, and
$$
\ap(t,x) = \frac{1}{(2\pi)^d}\int_{\Rd} e^{-t\left( |\xi|^{\alpha}+ a^{\beta}|\xi|^{\beta}\right)} e^{-ix\cdot\xi}d\xi,
\quad x\in\Rd,\quad t>0.
$$
For $t \le 0$ we put $\ap(t,x) = 0$. The convolution semigroup $\ap(t,x)$ has $\Delta^{\alpha /2} + a^{\beta} \Delta^{\beta/2}$ as its infinitesimal generator  (\cite{BF}, \cite{Y},
\cite{BB1}, \cite{CKS}). In particular, for $f\in C^\infty_c(\Rd)$, and $x\in \Rd$
we have
\begin{align*}
\left( \Delta^{\alpha/2} + a^{\beta} \Delta^{\beta/2}\right) f(x)&=\lim_{t\to 0^+}\frac{1}{t}\int_{\Rd}\ap(t,y-x)(f(y)-f(x))\,dy \\
=&\lim_{\varepsilon \to 0^+}\int\limits_{|y|>\varepsilon} \left( \frac{\mathcal{A}_{d,-\alpha}}{|y|^{d+\alpha}}+\frac{a^{\beta}\mathcal{A}_{d,-\beta}}{|y|^{d+\beta}} \right) \big[ f(x+y)-f(x) \big] dy\,,
\end{align*}
where $\mathcal{A}_{d,\gamma}= \Gamma((d-\gamma)/2)/(2^{\gamma}\pi^{d/2}|\Gamma(\gamma/2)|)$. Let $\eta_t^a(u)$ be the density function of the sum of the $\alpha/2$- stable subordinator and $a^2$ times the $\beta/2$-stable subordinator. Let $g_t (x)= (4\pi t)^{-d/2}e^{-|x|/4t}$ be the d-dimensional Gaussian kernel. Then $\ap(t,x)$ can be expressed as
\begin{align*}
\ap (t,x)= \int_0^{\infty} g_u(x) \eta_t^a (u)\,du\,.
\end{align*}
Differentiating we obtain
\begin{equation}\label{eq:grad}
\nabla_x \ap (t,x)= -2\pi x \ap_{(d+2)}(t,\tilde{x})\,,
\end{equation}
where $\tilde{x}\in \RR^{d+2}$ is such that $|\tilde{x}|=|x|$ and $\ap_{(d+2)}$ stands for the function $\ap$ in dimension $d+2$ (see also \cite{BJ}).
It is crucial here to notice that $\ap (s,x,t,y)$ satisfies (\ref{assume1}) and (\ref{assume2}) for every $a\geq 0$.

In what follows we assume that $1< \beta < \alpha < 2$. This restriction emerges naturally for gradient perturbations of stable processes, although some of the results below (Lemma \ref{lem:Gp} and Remark \ref{rem:3P}) are true for any $0<\beta<\alpha<2$.

We first consider the case of $a=0$. Then $\ap(t,x)=p^0(t,x)$ simplifies to the transition density of the isotropic $\alpha$-stable L\'evy process. Gradient, or drift,  perturbations have been recently intensely studied for this process (see \cite{BJ,J1,J2,JS,TJ}). 
Theorem \ref{thm:4} takes the following form 

\begin{prop}
Let $b\in \pN (\eta,Q,p^0)$. If $\eta<1/2$, then there is a positive transition density $\widetilde{p^0}$ such that
\begin{equation*}
\int_s^{\infty}\int_{\Rd}\widetilde{p^0}(s,x,u,z)[\partial_u \phi(u,z) + \Delta^{\alpha/2}\phi(u,z) + b(u,z) \cdot \nabla_{z}\phi(u,z)]\,dzdu\,= -\phi(s,x)\,,
\end{equation*}
where $s\in \RR$, $x \in \Rd$ and $\phi \in C_{c}^{\infty}(\RR \times \Rd )$. 
\end{prop}
\noindent
We note that this result extends Theorem 1 in \cite{JS} to the wider class of drift functions from $\pN$.
We omit the proof as it is similar to that of \cite[Theorem 1]{JS}. We also remark that Theorem \ref{thm:2} in the present paper gives estimates for the gradient perturbations $\widetilde{p^0}$, if $b\in\pN$.

Now let $a>0$. By writing $f(x)\approx g(x)$ we mean that there is a number $0<C<\infty$ such that for every $x$ we have $C^{-1} f(x) \leq g(x) \leq C f(x)$. It is known that (see \cite{CKS})
\begin{align}\label{approx:ap}
\ap (t,x) \approx \left( t^{-d/ \alpha}\land (a^{\beta}t)^{-d/\beta}\right) \land \left( \frac{t}{|x|^{d+\alpha}} + \frac{a^{\beta}t}{|x|^{d+\beta}}\right)
\end{align}
on $(0,\infty)\times \Rd$,  and that the scaling property holds,
\begin{align}\label{eq:scaing}
\ap (t,x)= a^{\frac{\beta d}{\alpha - \beta}} p^1 (a^{\frac{\alpha \beta}{\alpha - \beta}}t,a^{\frac{\beta}{\alpha - \beta}}x)\,.
\end{align}
\noindent
To simplify the notation we denote
\begin{align*}
\widehat{\ap}(t,x)= (t^{-\frac{1}{\alpha}}\land (a^{\beta}t)^{-\frac{1}{\beta}})\ap (t,x)\,, \quad t>0,\, x\in \Rd.
\end{align*}
\begin{lem}\label{lem:Gp}
There exists a constant C such that for all $t>0$ and $x\in \Rd$,
\begin{align*}
|\nabla_x \ap(t,x)|\leq C \widehat{\ap}(t,x)\,.
\end{align*}
\end{lem}
\begin{proof}
By scaling we may assume that $a=1$. By (\ref{eq:grad}) and (\ref{approx:ap}) we have
\begin{align}\label{exp1}
|\nabla_x p^1(t,x)|\approx |x| \left( t^{-\frac{d+2}{\alpha}}\land t^{-\frac{d+2}{\beta}} \land \left( \frac{t}{|x|^{d+2+\alpha}} \lor \frac{t}{|x|^{d+2+\beta}}\right) \right)\,.
\end{align}
We claim that the right hand side of (\ref{exp1}) equals
\begin{align}\label{exp2}
|x| \left( t^{-\frac{2}{\alpha}}\land t^{-\frac{2}{\beta}} \land \frac{1}{|x|^2} \right) \left( t^{-\frac{d}{\alpha}}\land t^{-\frac{d}{\beta}} \land  \left( \frac{t}{|x|^{d+\alpha}} \lor \frac{t}{|x|^{d+\beta}}\right) \right)\,.
\end{align}
Indeed, the inequality
\begin{align*}
\left( \frac{t}{|x|^{d+\alpha}} \lor \frac{t}{|x|^{d+\beta}}\right) \leq  \left( t^{-\frac{d}{\alpha}}\land t^{-\frac{d}{\beta}} \right)\,,
\end{align*}
holds if and only if $|x|^{\alpha}\geq t$, $|x|^{\beta}\geq t$, $|x|^{\beta \frac{d+\alpha}{d+\beta}}\geq t$ and $|x|^{\alpha \frac{d+\beta}{d+\alpha}}\geq t$. But $\beta \leq \beta \frac{d+\alpha}{d+\beta}\leq \alpha$ and $\beta \leq \alpha \frac{d+\beta}{d+\alpha} \leq \alpha$, so these are equivalent to $|x|^{\alpha}\geq t$ and $|x|^{\beta}\geq t$, regardless of the dimension $d$. This proves the claim.\\
We now notice that
$$
|x| \left( t^{-\frac{2}{\alpha}}\land t^{-\frac{2}{\beta}} \land \frac{1}{|x|^2}\right) \leq \left( t^{-\frac{1}{\alpha}}\land t^{-\frac{1}{\beta}} \right)\,,
$$
which ends the proof.
\end{proof}
Now we prove Theorem \ref{thm:4}
\begin{proof}[Proof of Theorem \ref{thm:4}]
We note that by Lemma \ref{lem:Gp} and (\ref{approx:ap}) $\ap$ satisfies 
\begin{align*}
\frac{\partial}{\partial x_i} \int_s^{\infty} \int_{\Rd} \ap(s,x,r,z)\psi(r,z) dzdr = \int_s^{\infty} \int_{\Rd} \frac{\partial}{\partial x_i} \ap(s,x,r,z)\psi(r,z)dzdr\,,
\end{align*}
for any $\psi: \RR \times \Rd \to \RR$ such that $|\psi(s,x)|\leq c \ap(s,x,t_0,y_0)$ for some $c>0$, $t_0 \in \RR$, $y_0 \in \Rd$ and all $(s,x)\in \RR \times \Rd$. Moreover, by (\ref{approx:ap}) for any $\phi \in C_c^{\infty}(\RR\times \Rd)$ we can take $\psi(s,x)= \left( \Delta^{\alpha/2} + a^{\beta}\Delta^{\beta/2} \right)\phi (s,x) $.
Thus, the proof may be carried out as the proof of Theorem 1 in \cite{JS}.

\end{proof}

Next we will show some properties of the function $\ap(t,x)$ useful when dealing with conditions (\ref{def:coP}) or (\ref{def:coN}).

\begin{lem}[3P] \label{lem:3Pex}
There exists a constant C such that for all $0<u,r<\infty$ and $x,y \in \Rd$ we have
\begin{align}\label{ineq:3Pex}
\widehat{\ap}(u,x)\land \widehat{\ap}(r,y)\leq C \widehat{\ap}(u+r,x+y)\,.
\end{align}
\end{lem}
\begin{proof}
By (\ref{eq:scaing}) it suffices to consider only  $a=1$. We first notice that
\begin{align}
\left( u^{-\frac{d+1}{\alpha}}\right.  &\land  \left.u^{-\frac{d+1}{\beta}}\right)  \land \left( r^{-\frac{d+1}{\alpha}}\land r^{-\frac{d+1}{\beta}}\right) \nonumber\\
&\leq c \left( \left( u+r \right)^{-\frac{d+1}{\alpha}} \land \left(u+r \right)^{-\frac{d+1}{\beta}} \right) \nonumber\\
&= c \left( \left( u+r \right)^{-\frac{1}{\alpha}} \land \left( u+r \right)^{-\frac{1}{\beta}} \right) \left( \left( u+r \right)^{-\frac{d}{\alpha}} \land \left( u+r \right)^{-\frac{d}{\beta}} \right)\,. \label{eq:3Pex1}
\end{align}
Since 
\begin{align*}
\left( u^{1-\frac{1}{\alpha}}\land u^{1-\frac{1}{\beta}} \right) &\leq  \left(  (u+r)^{1-\frac{1}{\alpha}} \land (u+r)^{1-\frac{1}{\beta}} \right)\,,\\
\left( r^{1-\frac{1}{\alpha}}\land r^{1-\frac{1}{\beta}}\right) &\leq \left(  (u+r)^{1-\frac{1}{\alpha}} \land (u+r)^{1-\frac{1}{\beta}} \right)\,,
\end{align*}
we conclude that
\begin{align}
&\left( \frac{u^{1-\frac{1}{\alpha}} \land u^{1-\frac{1}{\beta}}}{|x|^{d+\alpha}} + \frac{u^{1-\frac{1}{\alpha}} \land u^{1-\frac{1}{\beta}}}{|x|^{d+\beta}} \right) \land \left( \frac{r^{1-\frac{1}{\alpha}} \land r^{1-\frac{1}{\beta}}}{|y|^{d+\alpha}} + \frac{r^{1-\frac{1}{\alpha}} \land r^{1-\frac{1}{\beta}}}{|y|^{d+\beta}} \right) \nonumber\\
\leq &\,   \left(  (u+r)^{1-\frac{1}{\alpha}} \land (u+r)^{1-\frac{1}{\beta}} \right) \left( \left( \frac{1}{|x|^{d+\alpha}} + \frac{1}{|x|^{d+\beta}} \right) \land \left( \frac{1}{|y|^{d+\alpha}} + \frac{1}{|y|^{d+\beta}} \right) \right) \nonumber\\
= &\,   \left(  (u+r)^{-\frac{1}{\alpha}} \land (u+r)^{-\frac{1}{\beta}} \right) \left( \left( \frac{u+r}{|x|^{d+\alpha}} \land \frac{u+r}{|y|^{d+\alpha}} \right) + \left( \frac{u+r}{|x|^{d+\beta}} \land \frac{u+r}{|y|^{d+\beta}} \right) \right) \nonumber\\
\leq &\, c \left(  (u+r)^{-\frac{1}{\alpha}} \land (u+r)^{-\frac{1}{\beta}} \right) \left( \frac{u+r}{|x+y|^{d+\alpha}} + \frac{u+r}{|x+y|^{d+\beta}} \right)\,. \label{eq:3Pex2}
\end{align}
Finally, by (\ref{eq:3Pex1}) and (\ref{eq:3Pex2}),
\begin{align*}
\widehat{p^1} (u,x)\land \widehat{p^1}(r,y) \approx &\left( u^{-\frac{d+1}{\alpha}}  \land u^{-\frac{d+1}{\beta}}\right)  \land \left( \frac{u^{1-\frac{1}{\alpha}} \land u^{1-\frac{1}{\beta}}}{|x|^{d+\alpha}} + \frac{u^{1-\frac{1}{\alpha}} \land u^{1-\frac{1}{\beta}}}{|x|^{d+\beta}} \right) \land \\
& \left( r^{-\frac{d+1}{\alpha}}\land r^{-\frac{d+1}{\beta}}\right) \land  \left( \frac{r^{1-\frac{1}{\alpha}} \land r^{1-\frac{1}{\beta}}}{|y|^{d+\alpha}} + \frac{r^{1-\frac{1}{\alpha}} \land r^{1-\frac{1}{\beta}}}{|y|^{d+\beta}} \right) \\
\leq &  c \left( \left( u+r \right)^{-\frac{1}{\alpha}} \land \left( u+r \right)^{-\frac{1}{\beta}} \right) \Bigg( \left( \left( u+r \right)^{-\frac{d}{\alpha}} \land \left( u+r \right)^{-\frac{d}{\beta}} \right) \land \\
&\left( \frac{u+r}{|x+y|^{d+\alpha}} + \frac{u+r}{|x+y|^{d+\beta}} \right) \Bigg) \\
\approx & \quad \widehat{p^1} (u+r,x+y)\,.
\end{align*}
\end{proof}
\noindent We note that for $a=0$, Lemma \ref{lem:3Pex} reduces to Lemma 4 from \cite{JS}.
\begin{rem}\label{rem:3P}
\begin{rm}
By exactly the same proof the inequality (\ref{ineq:3Pex}) is true for $1\leq \beta <\alpha <2$. For that range of $\alpha$ and $\beta$ it implies another $3P$-type inequality: for all $0<u,r< \infty$ and $x,y \in \Rd$
\begin{align}\label{ineq:3P}
\ap (u,x) \land \ap(r,y) \leq C \ap(u+r,x+y)\,.
\end{align}
However, (\ref{ineq:3P}) holds for any $0<\beta <\alpha <2$ by a proof simpler than that of Lemma \ref{lem:3Pex} (see \cite[Theorem 4]{BJ}). The details are left to the reader.
\end{rm}
\end{rem}

\begin{cor}\label{cor:php}
There exists a constant C such that for all $0<u,r<\infty$ and $x,y \in \Rd$ we have
\begin{align*}
\ap (u,x) \widehat{\ap}(r,y) \leq C \ap(u+r,x+y)\left( \widehat{\ap}(u,x) + \widehat{\ap}(r,y)\right)\,.
\end{align*}
\end{cor}
\begin{proof}
For any $a,b\geq 0$ we have $ab= (a\land b)(a \lor b)$ and $(a \lor b)\leq (a+b)$. We rewrite the right hand side, use Lemma \ref{lem:3Pex} and apply the inequality $\frac{1}{ \left( u^{-1/\alpha}\land u^{-1/\beta}\right)} \left( (u+r)^{-1/\alpha}\land(u+r)^{-1/\beta} \right)\leq 1$,
\begin{align*}
\frac{1}{\left( u^{-1/\alpha}\land u^{-1/\beta} \right)} \, \widehat{\ap}(u,x) \widehat{\ap} (r,y) \leq &\frac{C}{\left( u^{-1/\alpha}\land u^{-1/\beta} \right)} \, \widehat{\ap}(u+r, x+y) \\
&\left( \widehat{\ap}(u+r) +\widehat{\ap}(u+r) \right)\,.
\end{align*}
\end{proof}

\begin{rem}\label{rem:fac}
\begin{rm}
Lemma \ref{lem:Gp} and Corollary \ref{cor:php} give the existence of a constant $C$, such that for all $s<t$ and $x,y\in \Rd$,
\begin{align}
&\int_s^t \int_{\Rd}\ap (s,x,t,y) |b(u,z)||\nabla_z \ap (u,z,t,y)|\, dzdu\nonumber \\
\leq & \, C  \left( \int_s^t \int_{\Rd}\left( \widehat{\ap}(s,x,u,z) + \widehat{\ap}(u,z,t,y) \right) |b(u,z)|\, dzdu \right) \, \ap (s,x,t,y)\,, \label{ineq:fac}
\end{align}
where $\widehat{\ap}(s,x,t,y)= \widehat{\ap}(t-s, y-x)$,  $s<t$, $x,y\in\Rd$. The inequality (\ref{ineq:fac}) may be used for verifying that $b\in \pN(\eta,Q,\ap)$. 
\end{rm}
\end{rem}

We complete this section with two examples 
\begin{example}
\begin{rm}
Let $\gamma > 1$. Recall that $b(u,z)=b(z)$ belongs to the Kato class $\pK_{d}^{\gamma -1}$ if
\begin{align*}
\lim_{\varepsilon \to 0} \sup_{x \in \Rd} \int_{|z-x|<\varepsilon} |b(z)| |z-x|^{\gamma -(d+1)}\, dz = 0\,.
\end{align*}
Note that $\pK_d^{\beta - 1} \subset \pK_d^{\alpha -1}$. We have
\begin{align*}
\widehat{\ap}(t,x) &= (t^{-\frac{1}{\alpha}}\land (a^{\beta}t)^{-\frac{1}{\beta}}) \ap (t,x) \\
&\leq c \left( t^{-\frac{1}{\alpha}} p^0 (t,x) + (a^{\beta}t)^{-\frac{1}{\beta}} p_{(\beta)}(a^{\beta}t,x)\right)\,,
\end{align*}
where $c$ is a constant independent of $t>0$ and $x\in \Rd$, $p_{(\beta)}$ is the density function of the isotropic $\beta$-stable L$\rm{\acute{e}}$vy process. Thus, by Remark \ref{rem:fac} and Example 1 in \cite{JS}  we obtain that if $b\in \pK_d^{\beta -1}$, then $b\in \pN(\eta,Q_{\eta}, \ap)$ for any $\eta >0$.
\end{rm}
\end{example}

\begin{example}
\begin{rm}
Let $b(u,z)=b(z)$ be such that $|b(z)|=|z|^{ 1-\alpha +\varepsilon}$, for some $0<\varepsilon<\alpha-\beta$. Then $b\in \pK_{d}^{\alpha -1}$ and $b \notin \pK_{d}^{\beta -1}$. By integrating (\ref{approx:ap}) we get
\begin{align}
\int_0^t \widehat{p^1} (u,x)\, du &\approx \left( \frac{ t^{2-\frac{1}{\alpha}} \land t^{2-\frac{1}{\beta}}}{|x|^{d+
\alpha}} + \frac{t^{2-\frac{1}{\alpha}} \land t^{2-\frac{1}{\beta}}}{|x|^{d+\beta}}\right) \land \left( |x|^{\alpha -(d+1)} \land |x|^{\beta - (d+1)} \right) \nonumber \\
&\leq |x|^{\alpha -(d+1)}\,. \label{ineq:ex2}
\end{align}
Let $\varepsilon>0$ and $0<\delta<1$. We split the integral in (\ref{ineq:fac}) with $a=1$ into two: over $A=\{z\in \Rd : \, |z|< \delta \}$ and $B=\Rd \backslash A$. We choose $\delta$ small enough that by (\ref{ineq:ex2}) the integral over $A$ is less than $\varepsilon/2$ . When integrating over $B$, we use $|z|^{1-\alpha+\varepsilon}\leq \delta^{1-\alpha +\varepsilon}$ and take $h>0$ such that for $t-s<h$ the integral does not exeed $\varepsilon/2$. We have just shown that $b \in \pP(\eta,p^1)$ for any $\eta>0$.
\end{rm}
\end{example}
\section*{Acknowledgment} 
We would like to thank Krzysztof Bogdan for many helpful comments on this paper
\bibliographystyle{abbrv}
\bibliography{tdgp1}

\end{document}